\documentclass[a4paper,12pt]{article}

\usepackage{amsthm,amsfonts}
\usepackage[leqno]{amsmath}

\usepackage{url}
\usepackage{enumerate}
\usepackage{multirow}
\usepackage{array}

\theoremstyle{plain}
\newtheorem{lemma}{Lemma}
\newtheorem{theorem}[lemma]{Theorem}

\newtheorem{cor}[lemma]{Corollary}

\theoremstyle{definition}

\newtheorem{remark}[lemma]{Remark}

\newtheorem{ques}[lemma]{Question}

\title{On the smallest eigenvalues of the line graphs of some trees}

\author{
{\sc Akihiro MUNEMASA}\\
[1ex]
{\small 
Graduate School of Information Sciences,} \\
{\small 
Tohoku University, Sendai 980-8579, Japan} \\
{\small 
{\it E-mail address}: {\tt munemasa@math.is.tohoku.ac.jp}}\\
\\
{\sc Yoshio SANO}
\thanks{This work was supported by JSPS KAKENHI grant number 25887007.}\\
[1ex]
{\small 
Division of Information Engineering,} \\
{\small 
Faculty of Engineering, Information and Systems,} \\
{\small 
University of Tsukuba, Ibaraki 305-8573, Japan} \\
{\small 
{\it E-mail address}: {\tt sano@cs.tsukuba.ac.jp}}\\
\\
{\sc Tetsuji TANIGUCHI}
\thanks{This work was supported by JSPS KAKENHI grant number 25400217.}\\
[1ex]
{\small 
Department of Electronics and Computer Engineering,}\\
{\small 
Hiroshima Institute of Technology, Hiroshima 731-5193, Japan}\\
{\small 
{\it E-mail address}: {\tt t.taniguchi.t3@cc.it-hiroshima.ac.jp}}
} 

\date{}

\begin{document}

\maketitle

\begin{abstract}
In this paper, we study the characteristic polynomials 
of the line graphs of generalized Bethe trees. 
We give an infinite family of
such graphs sharing the same smallest eigenvalue.
Our family generalizes the family of coronas of complete graphs
discovered by Cvetkovi\'{c} and Stevanovi\'{c}. 
\end{abstract}




\newpage

\section{Introduction}\label{sec:001}

All graphs considered in this paper are finite, undirected and simple.
By an eigenvalue of a graph we mean an eigenvalue of its adjacency matrix.
It is well known that graphs 
whose smallest eigenvalue is greater than $-2$
are the line graphs of trees, or the line graphs of a unicyclic graph
with an odd cycle, certain generalized line graphs of trees, or
exceptional graphs arising from the root system $E_8$ (see
\cite{DC79}). 

In this paper, we focus on the line graphs of trees 
and study the smallest eigenvalue of a particular type of such graphs. 
Our research is motivated by a question raised
by Cvetkovi\'{c} and Stevanovi\'{c}. 
In \cite{CS}, 
it is shown that 
the sequence $\{ \lambda_{\min}(K_n \otimes K_q) \}_{n=1}^{\infty}$
is constant for a fixed integer $q$,  
where $\otimes$ denote the corona of graphs  
(see page 51 of \cite{CvDS} for a definition of corona 
and see also \cite{GM} for its generalization),
and $\lambda_{\min}$ denotes the smallest eigenvalue. 
Cvetkovi\'{c} and Stevanovi\'{c}
raised the following question: 

\begin{ques}
Do there exist other sequences of the line graphs 
of trees whose smallest eigenvalues are constant?
\end{ques}

In this paper, we give an answer for this question 
by 
giving a larger family of 
graph sequences of the line graphs 
of trees which have a constant smallest eigenvalue (Corollary \ref{cor:main}). 

For 
positive integers $ d_1=1$, $d_2\geq2,\dots,d_{k-1}\geq2$, $d_k\geq1$, 
we define a tree $ B(d_1,\dots,d_k) $ to be a
rooted tree with $k$ levels in which every vertex at level $j$ has degree
$d_{k-j+1}$. 
Note that 
$K_{n} \otimes K_{q}$ 
is isomorphic to the line graph 
$L(B(1,q,n))$ of the tree $B(1,q,n)$. 
The tree $B(d_1,\dots,d_k)$ is called a \emph{generalized Bethe tree}
(see \cite{RJ,RS}).

In the next section, 
we determine the characteristic polynomial of the line graph 
$L(B(d_1,\dots,d_k))$ of the tree $B(d_1,\dots,d_k)$ 
in a factored form, 
thereby showing that the smallest eigenvalue of $L(B(d_1,\dots,d_k))$ 
is independent of $d_k$ (Theorem \ref{thm:001g}).
We also show that the smallest eigenvalue of $L(B(d_1,\dots,d_k))$ 
has multiplicity $d_k-1$, and is a zero of a polynomial of degree $k-1$ 
(Theorem \ref{thm:main}). The characteristic polynomial of
$L(B(d_1,\dots,d_k))$ has also been determined by Rojo and Jim\'enez \cite{RJ}
using a different method, but our result gives more concrete information
about the smallest eigenvalue.

\section{The characteristic polynomial}\label{sec:1}

We denote by $\chi_G(\lambda)$ the characteristic polynomial of
the adjacency matrix $A(G)$ of a graph $G$, that is,
$\chi_G(\lambda)=\det(\lambda I -A(G))$.

Let $G$ and $H$ be rooted graphs with roots $u$ and $v$, respectively.  
We denote by $G \cdot H$ the graph obtained from $G$ and $H$ 
by identifying the vertices $u$ and $v$.  

\begin{lemma}[Schwenk {\cite[Corollary 2b]{LNM406}}]\label{lem:Schwenk2}
Let $G$ and $H$ be rooted graphs with roots $u$ and $v$, respectively. 
Then
\[
\chi_{G \cdot H}(\lambda) 
= \chi_{G-u}(\lambda)\chi_{H}(\lambda)
+\chi_{G}(\lambda)\chi_{H-v}(\lambda)
-\lambda\chi_{G-u}(\lambda)\chi_{H-v}(\lambda). 
\]
\end{lemma}

\begin{lemma}[Schwenk {\cite[Theorem 5]{LNM406}}]\label{lem:Schwenk}
Let $G_0$ be a graph of order $p$, and let $H$ be a rooted graph
with root $v$. Let $G$ be the graph obtained by
attaching $H$ to each vertex $u$ of $G_0$ by identifying $u$ with $v$.
Then
\[
\chi_G(\lambda)=\chi_{H-v}(\lambda)^p
\chi_{G_0}\left(\frac{\chi_H(\lambda)}{\chi_{H-v}(\lambda)}\right).
\]
\end{lemma}

\begin{lemma}\label{lem:int}
Let $H$ be a rooted graph 
with root $v$,
and let $s\geq2$ be an integer.
Let $G$ be the graph obtained by
attaching $H$ to each vertex $u$ of 
the complete graph $K_s$ by identifying $u$ with $v$.
Then 
\begin{equation}\label{chiG}
\chi_G(\lambda)=
(\chi_H(\lambda)-(s-1)\chi_{H-v}(\lambda))
(\chi_H(\lambda)+\chi_{H-v}(\lambda))^{s-1},
\end{equation}
and the smallest eigenvalue of $G$ is the smallest zero of the
polynomial $\chi_H(\lambda)+\chi_{H-v}(\lambda)$.
In particular, the smallest eigenvalue of $G$ is independent of $s$.
\end{lemma}

\begin{proof}
Setting $G_0=K_s$ in Lemma~\ref{lem:Schwenk} gives (\ref{chiG}).
Let $a$ and $b$ denote the smallest zero of the polynomials
$\chi_H(\lambda)$ and
$\chi_{H-v}(\lambda)$, respectively.
Then by interlacing, we have $a\leq b$.

First consider the case where the number of vertices of $H$ is odd.
Then $\chi_H(\lambda)$ is a monic polynomial of odd degree, so
$\chi_H(\lambda)<0$ for $\lambda<a$. Since 
$\chi_{H-v}(\lambda)$ is a monic polynomial of even degree and
$a\leq b$, we have $\chi_{H-v}(\lambda)\geq0$ for $\lambda<a$.
Thus $\chi_H(\lambda)-(s-1)\chi_{H-v}(\lambda)<0$ for $\lambda<a$, 
implying that the smallest zero of the polynomial
$\chi_H(\lambda)-(s-1)\chi_{H-v}(\lambda)$ is at least $a$.
On the other hand,
since $\chi_H(\lambda)+\chi_{H-v}(\lambda)$
is a monic polynomial of odd degree and
$\chi_H(a)+\chi_{H-v}(a)=\chi_{H-v}(a)\geq0$, the smallest zero
of $\chi_H(a)+\chi_{H-v}(a)$ is at most $a$. 
This proves the assertion in this case.

The case where the number of vertices of $H$ is even can be
proved similarly.
\end{proof}

\begin{remark}
Note that if $H$ is a complete graph $K_t$, then the graph $G$ in
Lemma~\ref{lem:int} is the corona of $K_s$ and $K_{t-1}$. 
The spectrum of the corona $G_1\otimes G_2$ of two graphs 
$G_1$ and $G_2$ is given in \cite{corona}, under the assumption that
$G_2$ is regular.
\end{remark}

\begin{remark}
We remark that the statement of Lemma~\ref{lem:int} cannot be
strengthened to claim that the smallest eigenvalue of $G$ is not
a zero of the other factor $\chi_H(\lambda)-(s-1)\chi_{H-v}(\lambda)$.
Indeed, as pointed out in \cite[Remark 9]{GKMST}, if $H$
is the line graph of the Dynkin diagram $E_6$, then choosing
the vertex $r$ appropriately, we can make $\chi_H$ and $\chi_{H-v}$
to have a common smallest zero. In this case, the smallest eigenvalue
of $G$ is a common zero of 
$\chi_H(\lambda)-(s-1)\chi_{H-v}(\lambda)$ and
$\chi_H(\lambda)+\chi_{H-v}(\lambda)$.
\end{remark}

\begin{lemma}\label{lem:int2}
Let $H$ be a rooted graph with root $v$, and let $s\geq2$ be an integer.
Let $G$ be the graph obtained by 
attaching $H$ to all vertices $u$ of the complete graph $K_s$ except one, 
by identifying $u$ with $v$.
Then 
\begin{equation}\label{chiG2}
\chi_G(\lambda)= 
(\lambda\chi_H(\lambda)-((s-2)\lambda+(s-1))\chi_{H-v}(\lambda))
(\chi_H(\lambda)+\chi_{H-v}(\lambda))^{s-2}.
\end{equation}
\end{lemma}

\begin{proof}
Let $u_0$ be the vertex of $K_s$ 
for which the graph $H$ were not attached to construct $G$. 
Consider the graph $G$ as a rooted graph with root $u_0$. 
Recall that $G \cdot H$ is the graph obtained from $G$ and $H$ 
by identifying $u_0$ and $v$. 
Since $G \cdot H$ is the same as 
the graph obtained by
attaching $H$ to each vertex $u$ of 
the complete graph $K_s$ by identifying $u$ with $v$, 
it follows from 
Lemma~\ref{lem:int} that 
\[
\chi_{G \cdot H}(\lambda)=
(\chi_H(\lambda)-(s-1)\chi_{H-v}(\lambda))
(\chi_H(\lambda)+\chi_{H-v}(\lambda))^{s-1}. 
\]
Since $G - u_0$ is 
the graph obtained by
attaching $H$ to each vertex $u$ of 
the complete graph $K_{s-1}$ by identifying $u$ with $v$, 
it follows from 
Lemma \ref{lem:int} that 
\[
\chi_{G -u_0}(\lambda)=
(\chi_H(\lambda)-(s-2)\chi_{H-v}(\lambda))
(\chi_H(\lambda)+\chi_{H-v}(\lambda))^{s-2}.
\]
By Lemma \ref{lem:Schwenk2}, we have 

\[
\chi_{G}(\lambda)
=
\frac{\chi_{G \cdot H}(\lambda)
+(\lambda\chi_{H-v}(\lambda)-\chi_{H}(\lambda))\chi_{G-u_0}(\lambda)} 
{\chi_{H-v}(\lambda)}. 
\]
Now the lemma follows from the above three equations. 
\end{proof}

\begin{lemma}\label{lem:attach}
Let $d_1=1$, $d_2\geq2,\dots,d_k\geq2$ be integers, where
$k\geq2$. Regard $H=L(B(d_1,\dots,d_{k-1},1))$ as a rooted graph
whose root is the unique edge $e$ incident with the root of
$B(d_1,\dots,d_{k-1},1)$, and set $H'=H-e$.
Then $H'$ is isomorphic to
$L(B(d_1,\dots,d_{k-2},d_{k-1}-1))$, and the following statements hold.
\begin{enumerate}
\item[{\rm(i)}]
If $d_k\geq2$, then 
\[
\chi_{L(B(d_1,\dots,d_k))}(\lambda)=
(\chi_H(\lambda)-(d_k-1)\chi_{H'}(\lambda))
(\chi_H(\lambda)+\chi_{H'}(\lambda))^{d_k-1},
\]
and the smallest eigenvalue of $L(B(d_1,\dots,d_k))$ is
independent of $d_k$.
\item[{\rm(ii)}]
\begin{align*}
\chi_{L(B(d_1,\dots,d_k,1))}(\lambda)
&=
(\lambda\chi_H(\lambda)-((d_k-2)\lambda+(d_k-1))\chi_{H'}(\lambda))
\\&\quad\cdot
(\chi_H(\lambda)+\chi_{H'}(\lambda))^{d_k-2}.
\end{align*}
\end{enumerate}
\end{lemma}
\begin{proof}
Clearly, $H'$ is isomorphic to
$L(B(d_1,\dots,d_{k-2},d_{k-1}-1))$.

(i) The graph $L(B(d_1,\dots,d_k))$ is obtained by attaching $H$ to
each vertex $u$ of the complete graph $K_{d_k}$ by identifying
$u$ with $e$. The result follows from Lemma~\ref{lem:int}.

(ii) The graph $L(B(d_1,\dots,d_k,1))$ is obtained by attaching $H$ to
all vertices $u$ of the complete graph $K_{d_k}$ except one,
by identifying
$u$ with $e$. The result follows from Lemma~\ref{lem:int2}.
\end{proof}

For the remainder of this section, 
we let $k\geq2$ be an integer, 
and let $ d_1=1$, $d_2\geq2,\dots,d_{k-1}\geq2$, $d_k\geq1$ be integers.
Set 
\begin{align*} 
m_l&=d_k\prod_{i=l+1}^{k-1}(d_{i}-1) \quad ( l = 0,1, \ldots, k-1),\\
m_k&=1,\\
\sigma_l&=m_l-m_{l+1} \quad ( l = 0,1, \ldots, k-1),\\
\sigma_k&=1.
\end{align*}
By convention, we have 
$\sigma_{k-1}=d_k-1$. 
We define polynomials $g_{i}(\lambda) $ ($i=0,1,\dots,k$) by
\begin{align*}
g_{i}(\lambda)&=(\lambda+2-d_{i})g_{i-1}(\lambda)-(d_{i}-1)g_{i-2}(\lambda) 
\quad(i=2,\dots,k-1),\\
g_{k}(\lambda)&=(\lambda+2-d_{k})g_{k-1}(\lambda)-d_{k}g_{k-2}(\lambda) 
\end{align*}
with seed values 
\[
g_0(\lambda)=1, \qquad
g_1(\lambda)=\lambda+1. 
\]
We note that these polynomials are essentially same as
the polynomials $P_j(\lambda)$ defined in \cite{RS} in the
sense that $g_j(\lambda)=P_j(\lambda+2)$. Also, it is easy to
show by induction that 
$g_{i}+g_{i-1}$ is divisible by $\lambda+2$ for $i=1,\dots,k-1$,
so $g_k$ is divisible by $\lambda+2$.

The formula (\ref{eq:char}) in the 
following theorem is due to Rojo and Jim\'enez \cite{RJ}.

\begin{theorem}\label{thm:001g}
Let $d_1=1$, $d_2\geq2,\dots,d_{k-1}\geq2$, $d_k\geq1$ be integers, where
$k\geq2$.
Then the characteristic polynomial of the line graph $L(B(d_1,\dots,d_k))$ 
of the generalized Bethe tree $B(d_1,\dots,d_k)$ is 
\begin{equation}\label{eq:char}
\chi_{L(B(d_1,\dots,d_k))}(\lambda)=
\frac{1}{\lambda+2}
\prod_{l=1}^{k} g_l(\lambda)^{\sigma_{l}}.
\end{equation}
The smallest eigenvalue of $L(B(d_1,\dots,d_k))$ 
depends only on
$d_2,\dots,d_{k-1}$, and is independent of $d_k$, 
as long as $d_k\geq2$.
\end{theorem}
\begin{proof}
We prove the formula (\ref{eq:char}) by induction on $k$. If $k=2$, then
the right-hand side is 
\[
(\lambda+1-d_2)(\lambda+1)^{d_2-1}
\]
which is $\chi_{K_{d_2}}(\lambda)$.
Since $K_{d_2}=L(B(d_1,d_2))$, (\ref{eq:char}) holds.

Suppose $k\geq3$ and that the assertion holds for $k-1$.
We first establish the formula (\ref{eq:char}) for the case where $d_k=1$.
If $k=3$, then $L(B(d_1,d_2,1))=K_{d_2}$, and (\ref{eq:char}) holds since
$g_3=(\lambda+2)(\lambda+1)(\lambda+1-d_2)$.
If $k\geq4$, then
regard $H=L(B(d_1,\dots,d_{k-2},1))$ as a rooted graph whose root
is the unique edge $e$ incident with the root of 
$B(d_1,\dots,d_{k-2},1)$, and set $H'=H-e$.
Then by Lemma~\ref{lem:attach}(ii), we have
\begin{align}
\label{8ii}
\chi_{L(B(d_1,\dots,d_{k-1},1))}(\lambda)&=
(\lambda\chi_H(\lambda)-((d_{k-1}-2)\lambda+(d_{k-1}-1))\chi_{H'}(\lambda))
\\&\quad\cdot
(\chi_H(\lambda)+\chi_{H'}(\lambda))^{d_{k-1}-2}.
\notag
\end{align}
Also, $H'$ is isomorphic to
$L(B(d_1,\dots,d_{k-3},d_{k-2}-1))$ by Lemma~\ref{lem:attach}.

Let $\{g_l\}_{l=0}^{k}$, $\{\hat{g}_l\}_{l=0}^{k-1}$ and 
$\{\check{g}_l\}_{l=0}^{k-2}$
be the polynomials associated to the sequences
$(d_1,\dots,d_{k-1},1)$, $(d_1,\dots,d_{k-2},1)$, and $(d_1,\dots,d_{k-3},d_{k-2}-1)$,
respectively. Then
\begin{align*}
g_0&=\hat{g}_0=\check{g}_0=1,\\
g_1&=\hat{g}_1=\check{g}_1=\lambda+1,\\
&\vdots\\
g_{k-3}&=\hat{g}_{k-3}=\check{g}_{k-3},\\
g_{k-2}&=\hat{g}_{k-2}=\check{g}_{k-2}-g_{k-3},\\
g_{k-1}&=\hat{g}_{k-1}-(d_{k-1}-1)g_{k-2}-(d_{k-1}-2)g_{k-3}.
\end{align*}
Let
\begin{align*}
\hat{\sigma}_l&=
\begin{cases}
\prod_{i=l+1}^{k-2}(d_i-1)-\prod_{i=l+2}^{k-2}(d_i-1)
&\text{if $0\leq l\leq k-3$,}\\
0&\text{if $l=k-2$,}\\
1&\text{if $l=k-1$}
\end{cases}
\\&=
\begin{cases}
\frac{\sigma_l}{d_{k-1}-1}&\text{if $0\leq l\leq k-3$,}\\
0&\text{if $l=k-2$,}\\
1&\text{if $l=k-1$,}
\end{cases}
\displaybreak[0]\\
\check{\sigma}_l&=
\begin{cases}
(d_{k-2}-1)(\prod_{i=l+1}^{k-3}(d_i-1)-\prod_{i=l+2}^{k-3}(d_i-1))
&\text{if $0\leq l\leq k-4$,}\\
d_{k-2}-2&\text{if $l=k-3$,}\\
1&\text{if $l=k-2$}
\end{cases}
\\&=
\begin{cases}
\frac{\sigma_{l}}{d_{k-1}-1}
&\text{if $0\leq l\leq k-3$,}\\
1&\text{if $l=k-2$.}
\end{cases}
\end{align*}
By induction,
\begin{align}\label{chi1H}
\chi_{H}(\lambda)&=
\frac{1}{\lambda+2}
\prod_{l=1}^{k-1}\hat{g}_l^{\hat{\sigma}_l}
\\&=
\frac{1}{\lambda+2}
(g_{k-1}+(d_{k-1}-1)g_{k-2}+(d_{k-1}-2)g_{k-3})
\prod_{l=1}^{k-3}g_l^{\sigma_l/(d_{k-1}-1)},
\notag\\
\chi_{H'}(\lambda)&=
\frac{1}{\lambda+2}
\prod_{l=1}^{k-2}\check{g}_l^{\check{\sigma}_l}
\label{chi1H'}\\&=
\frac{1}{\lambda+2}
(g_{k-2}+g_{k-3})\prod_{l=1}^{k-3}g_l^{\sigma_l/(d_{k-1}-1)}.
\notag
\end{align}
Since
\begin{align*}
&\lambda\chi_H(\lambda)-((d_{k-1}-2)\lambda+(d_{k-1}-1))\chi_{H'}(\lambda)
\\&=
(\lambda(g_{k-1}+(d_{k-1}-1)g_{k-2}+(d_{k-1}-2)g_{k-3})
\\&\qquad
-((d_{k-1}-2)\lambda+(d_{k-1}-1))(g_{k-2}+g_{k-3}))
\frac{\prod_{l=1}^{k-3} g_l^{\sigma_l/(d_{k-1}-1)}}{\lambda+2}
\\&=
(\lambda g_{k-1}+(\lambda-d_{k-1}+1)g_{k-2}-(d_{k-1}-1)g_{k-3})
\frac{\prod_{l=1}^{k-3} g_l^{\sigma_l/(d_{k-1}-1)}}{\lambda+2}
\\&=
((\lambda+1) g_{k-1}-g_{k-2})
\frac{\prod_{l=1}^{k-3} g_l^{\sigma_l/(d_{k-1}-1)}}{\lambda+2}
\end{align*}
and
\begin{align*}
\chi_H(\lambda)+\chi_{H'}(\lambda)&=
(g_{k-1}+(d_{k-1}-1)g_{k-2}+(d_{k-1}-2)g_{k-3}+(g_{k-2}+g_{k-3}))
\\&\quad\cdot
\frac{\prod_{l=1}^{k-3} g_l^{\sigma_l/(d_{k-1}-1)}}{\lambda+2}
\\&=
(g_{k-1}+d_{k-1}g_{k-2}+(d_{k-1}-1)g_{k-3})
\frac{\prod_{l=1}^{k-3} g_l^{\sigma_l/(d_{k-1}-1)}}{\lambda+2}
\\&=
(\lambda+2)g_{k-2}
\frac{\prod_{l=1}^{k-3} g_l^{\sigma_l/(d_{k-1}-1)}}{\lambda+2}
\\&=
g_{k-2}\prod_{l=1}^{k-3} g_l^{\sigma_l/(d_{k-1}-1)},
\end{align*}
substitution of (\ref{chi1H}) and (\ref{chi1H'}) into
(\ref{8ii}) gives 
\begin{align*}
&\chi_{L(B(d_1,\dots,d_{k-1},1))}
\\&=
((\lambda+1) g_{k-1}-g_{k-2})
\frac{\prod_{l=1}^{k-3} g_l^{\sigma_l/(d_{k-1}-1)}}{\lambda+2}
\left(
g_{k-2}\prod_{l=1}^{k-3} g_l^{\sigma_l/(d_{k-1}-1)}
\right)^{d_{k-1}-2}
\\&=
((\lambda+1) g_{k-1}-g_{k-2})
g_{k-2}^{d_{k-1}-2}
\frac{\prod_{l=1}^{k-3} g_l^{\sigma_l}}{\lambda+2}
\\&=
\frac{1}{\lambda+2}((\lambda+1)g_{k-1}-g_{k-2})
\prod_{l=1}^{k-2}g_l^{\sigma_l}.
\end{align*}
This establishes (\ref{eq:char}) for the case $d_k=1$.

Next we consider the case where $d_k\geq2$.
Let $\{g_l\}_{l=0}^{k}$, $\{\hat{g}_l\}_{l=0}^{k}$ and 
$\{\check{g}_l\}_{l=0}^{k-1}$
be the polynomials associated to the sequences
$(d_1,\dots,d_k)$, $(d_1,\dots,d_{k-1},1)$, and $(d_1,\dots,d_{k-2},d_{k-1}-1)$,
respectively. Then
\begin{align*}
g_0&=\hat{g}_0=\check{g}_0=1,\\
g_1&=\hat{g}_1=\check{g}_1=\lambda+1,\\
&\vdots\\
g_{k-2}&=\hat{g}_{k-2}=\check{g}_{k-2},\\
g_{k-1}&=\hat{g}_{k-1}=\check{g}_{k-1}-g_{k-2},\\
g_{k}&=\hat{g}_{k}-(d_k-1)(g_{k-1}+g_{k-2}).
\end{align*}
Let
\begin{align*}
\hat{\sigma}_l&=
\begin{cases}
\prod_{i=l+1}^{k-1}(d_i-1)-\prod_{i=l+2}^{k-1}(d_i-1)
&\text{if $0\leq l\leq k-2$,}\\
0&\text{if $l=k-1$,}\\
1&\text{if $l=k$}
\end{cases}
\\&=
\begin{cases}
\frac{\sigma_l}{d_k}&\text{if $0\leq l\leq k-2$,}\\
0&\text{if $l=k-1$,}\\
1&\text{if $l=k$,}
\end{cases}
\\
\check{\sigma}_l&=
\begin{cases}
\prod_{i=l+1}^{k-1}(d_i-1)-\prod_{i=l+2}^{k-1}(d_i-1)
&\text{if $0\leq l\leq k-2$,}\\
1&\text{if $l=k-1$}
\end{cases}
\\&=
\begin{cases}
\frac{\sigma_{l}}{d_k}
&\text{if $0\leq l\leq k-2$,}\\
1&\text{if $l=k-1$.}
\end{cases}
\end{align*}
By the first part,
\begin{align}\label{chi1g}
\chi_{L(B(d_1,\dots,d_{k-1},1))}(\lambda)&=
\frac{1}{\lambda+2}
\prod_{l=1}^{k}\hat{g}_l^{\hat{\sigma}_{l}}
\\&=
\frac{1}{\lambda+2}(g_{k}+(d_k-1)(g_{k-1}+g_{k-2}))
\prod_{l=1}^{k-2} g_l^{\sigma_{l}/d_k},
\notag\displaybreak[0]\\
\intertext{and by induction,}
\label{chi2g}
\chi_{L(B(d_1,\dots,d_{k-2},d_{k-1}-1))}(\lambda)&=
\frac{1}{\lambda+2}
\prod_{l=1}^{k-1}\check{g}_l^{\check{\sigma}_{l}}
\\&=
\frac{1}{\lambda+2}(g_{k-1}+g_{k-2})
\prod_{l=1}^{k-2} g_l^{\sigma_{l}/d_k}.
\notag
\end{align}
Let $H=L(B(d_1,\dots,d_{k-1},1))$ and let $H'$ be as in
Lemma~\ref{lem:attach}. Then substitution of
(\ref{chi1g}) and (\ref{chi2g}) into the formula in
Lemma~\ref{lem:attach}(i) gives
\begin{align*}
\chi_{L(B(d_1,\dots,d_k))}(\lambda)
&=g_k(g_k+d_k(g_{k-1}+g_{k-2}))^{d_k-1}
\frac{\prod_{l=1}^{k-2}g_l^{\sigma_{l}}}{(\lambda+2)^{d_k}}
\\&=
g_k((\lambda+2)g_{k-1})^{d_k-1}
\frac{\prod_{l=1}^{k-2}g_l^{\sigma_{l}}}{(\lambda+2)^{d_k}}
\\&=
\frac{\prod_{l=1}^{k}g_l^{\sigma_{l}}}{\lambda+2}.
\end{align*}
Also, Lemma~\ref{lem:attach} implies that
the smallest eigenvalue of $L(B(d_1,\dots,d_{k-1},d_k))$ 
is independent of $d_k$.
\end{proof}

\section{A polynomial having the smallest eigenvalue as a zero}

Note that Lemma~\ref{lem:int} implies that the smallest eigenvalue
of $L(B(d_1,\dots,d_k))$ is the smallest zero of the 
polynomial
\begin{equation}\label{factor2}
\chi_{L(B(d_1,\dots,d_{k-1},1))}(\lambda)
+\chi_{L(B(d_1,\dots,d_{k-1}))}(\lambda)=
(g_k+d_k(g_{k-1}+g_{k-2}))\prod_{l=1}^{k-2} g_l^{\sigma_{l}/d_k}.
\end{equation}
It does not, however, tell which of the factors in the
right-hand side of (\ref{factor2}) contains the
smallest eigenvalue as a zero. Our next task is to show
that the smallest eigenvalue is a zero of $g_k+d_k(g_{k-1}+g_{k-2})$, but not
a zero of $g_l$ for $1\leq l\leq k-1$.

\begin{lemma}\label{lem:monotone} 
For $i=1,2, \ldots, k$, 
all the zeros of the polynomial $g_i(\lambda)$ 
are simple. 
\end{lemma}

\begin{proof}
By Favard's Theorem 
(see Theorem 4.4 on page 21 of \cite{Chihara78}), 
$g_i(\lambda)$ $(i=0,1, \ldots,k)$ are orthogonal polynomials. 
By Theorem 5.2 on page 27 of \cite{Chihara78}, 
the lemma holds. 
\end{proof}

\begin{lemma}\label{lem:ind}
For $i=1,2, \dots, k$, let $\gamma_i$ 
be the smallest zero of the polynomial $g_i(\lambda)$, and 
let $\beta$ be the second smallest zero
of $g_k(\lambda)$.
Then $\gamma_1 > \gamma_2 > \cdots > \gamma_k=-2$
and $\beta>\gamma_{k-1}$.
\end{lemma}
\begin{proof}
We have already noted $-2$ is a zero of $g_k(\lambda)$.
Since the smallest eigenvalue of the line graph of a tree is greater than $-2$,
Theorem~\ref{thm:001g} implies $\gamma_k=-2<\beta$.
Now the inequalities
follow from the separation theorem for the zeros 
(see Theorem 5.3 on page 28 of \cite{Chihara78}).
\end{proof}

\begin{theorem}\label{thm:main}
For integers 
$d_1=1$, $d_2\geq2,\dots,d_{k}\geq2$,
the smallest eigenvalue of $L(B(d_1,\dots,d_k))$ 
is the smallest zero of the polynomial
$g_{k-1}(\lambda)$, which is not a zero of any of the polynomials
\[
\frac{g_{k}(\lambda)}{\lambda+2},
g_1(\lambda),\dots,g_{k-2}(\lambda).
\]
In particular, 
the multiplicity of the smallest eigenvalue of $L(B(d_1,\dots,d_k))$ 
is $d_k - 1$. 
\end{theorem}

\begin{proof}
The first statement follows from Theorem~\ref{thm:001g} and
Lemma~\ref{lem:ind}.
By Theorem~\ref{thm:001g}, 
the multiplicity of the smallest eigenvalue is $d_k-1$, 
since $g_{k-1}$ has only simple zeros by Lemma~\ref{lem:monotone}.
\end{proof}

\begin{cor}\label{cor:main}
Let $G_i=L(B(d_1,\dots,d_{k-1},i))$ for $i=1,2,\dots$.
Then the sequence $\{\lambda_{\min}(G_i)\}_{i=1}^{\infty}$ is constant. 
\end{cor}

\begin{proof}
Since $g_{k-1}$ is independent of $d_k$, 
the corollary follows from Theorem~\ref{thm:main}. 
\end{proof}

\begin{remark}\label{rem:s2}
We note that the smallest eigenvalue of 
$L(B(d_1,\dots,d_{k-1},d_k))$ 
does depend on $d_{k-1}$.
Indeed, let 
$g_{k-1}^{(i)}(\lambda)$ be the $(k-1)$st polynomial associated with the sequence
$d_1=1,d_2,\dots,d_{k-1}^{(i)},d_k$
where $d_{k-1}^{(1)}\neq d_{k-1}^{(2)}$.
Suppose that $\gamma$ is the common smallest zero of 
$g_{k-1}^{(i)}(\lambda)$ for $i=1,2$. Since
\[
(\gamma+1)g_{k-2}(\gamma)
-(d_{k-1}^{(i)}-1)(g_{k-2}(\gamma)+g_{k-3}(\gamma))=0,
\]
we have
\[
(d_{k-1}^{(1)}-d_{k-1}^{(2)})(g_{k-2}(\gamma)+g_{k-3}(\gamma))=0.
\]
Since $d_{k-1}^{(1)}\neq d_{k-1}^{(2)}$, we have
$g_{k-2}(\gamma)+g_{k-3}(\gamma)=0$, and hence
$(\gamma+1)g_{k-2}(\gamma)=0$. 
By Lemma~\ref{lem:ind}, we have $g_{k-2}(\gamma)\neq0$,
while the same lemma also 
implies $\gamma<-1$
since $-1$ is the zero of $g_1(\lambda)$.
This is a contradiction.
\end{remark}


\end{document}